\documentclass[11pt]{amsart}
\usepackage[utf8]{inputenc}
\usepackage{amsmath}
\usepackage{palatino}
\usepackage{xcolor}
\usepackage{amsfonts}
\usepackage{amsthm}
\usepackage{amssymb}
\usepackage{amscd}
\usepackage[all]{xy}
\usepackage{enumerate}
\usepackage{algorithm}
\usepackage{cleveref}
\theoremstyle{plain}

\newtheorem{thm}{Theorem}[section]

\newtheorem{prop}[thm]{Proposition}

\newtheorem{cor}[thm]{Corollary}

\theoremstyle{definition}
\newtheorem{defn}[thm]{Definition}
\newtheorem{definition}[thm]{Definition}

\newtheorem*{ackn}{Acknowledgment}

\newtheorem{rmk}[thm]{Remark}

\newcommand{\C}{\mathbb{C}}

\newcommand{\mC}{\mathbb{C}}

\newcommand{\R}{\mathbb{R}}

\newcommand{\mP}{\mathbb{P}}
\newcommand{\p}{\mathbb{P}}

\newcommand{\mG}{\mathbb{G}}
\newcommand{\G}{\mathbb{G}}

\newcommand{\LL}{\mathcal{L}}
\newcommand{\codim}{\mathop{\rm codim}\nolimits}

\newcommand{\GL}{\mathop{\rm GL}\nolimits}
\newcommand{\rank}{\mathop{\rm rank}\nolimits}

\newcommand{\Sym}{\mathop{\rm Sym}\nolimits}
\newcommand{\Eig}{\mathop{\rm Eig}\nolimits}
\newcommand{\Bl}{\mathop{\rm Bl}\nolimits}
\newcommand{\Ima}{\mathop{\rm Im}\nolimits}
\newcommand*{\de}{\partial}

\newcommand{\rvline}{\hspace*{-\arraycolsep}\vline\hspace*{-\arraycolsep}}

\title{Equations of tensor eigenschemes}

\author{Valentina Beorchia, Francesco Galuppi and Lorenzo Venturello}

\address{Dipartimento di Matematica e Geoscienze, Universit\`a di Trieste\\
via Valerio 12/1, 34126 Trieste, Italy}
\email{beorchia@units.it}

\address{Institute of Mathematics of the Polish Academy of Sciences\\
\'Sniadeckich 8, 00-656 Warsaw, Poland}
\email{francesco.galuppi@impan.pl}

\address{Royal Institute of Technology\\
Lindstedtsv\"gen 25, SE-100 44 Stockholm, Sweden}
\email{lven@kth.se}

\subjclass[2020]{Primary 15A18, 14M12; Secondary 13P25, 15A69}
\begin{document}

\maketitle

\begin{abstract} We study schemes of tensor eigenvectors from an algebraic and geometric viewpoint. We characterize determinantal defining equations of such eigenschemes via linear equations in their coefficients, both in the general and in the symmetric case. We give a geometric necessary condition for a 0-dimensional scheme to be an eigenscheme.
\end{abstract}

\section{Introduction}

Tensor spectral theory is an active area of research. Its applications range from hypergraphs theory to quantum physics, from dynamical systems to polynomial optimization, from signal processing to medical imaging. A reference for these and many more applications is \cite{QiChenChen2018}. In this article we study \emph{eigenvectors of tensors}. There are several notions of tensor eigenvectors as introduced independently in \cite{Lim} and \cite{Qi}, all extending the case of matrices. In the present paper we shall focus on the definition \cite[Section 2.2.1]{QZ}. In this context it is not
restrictive to consider partially symmetric tensors of order $d$, that is on elements of $\Sym^{d-1}\C^{n+1}\otimes\C^{n+1}$. 

If we choose a basis for $\C^{n+1}$, then we can identify a partially symmetric tensor $T$ with a tuple $(g_0,\dots,g_n)$ of homogeneous polynomials of degree $d-1$. With such a choice, an \emph{eigenvector} of $T$ is a nonzero vector $v\in\C^{n+1}$ such that
\begin{equation}\label{eq: definizione di autovettore}
(g_0(v), \ldots, g_n(v))=\lambda v\end{equation} for some constant $\lambda$. We observe that such vectors are called \emph{E-eigenvectors} in \cite[Section 2.2.1]{QZ}. Since  
property \eqref{eq: definizione di autovettore} is preserved under scalar multiplication, it is natural to regard eigenvectors as points in $\p^n$; hence the name \emph{eigenpoints} instead of eigenvectors.

Motivated by \cite[Question 16]{RS}, we are interested in the geometry of eigenpoints of tensors in the framework of complex projective geometry. One of the reasons is that being an eigenpoint of a tensor is an algebraic condition, described by the vanishing of minors of a suitable matrix - see \Cref{def:eigenscheme}. Hence, eigenpoints of a tensor form a scheme.

There are several challenging questions about eigenschemes, their dimension, their degree, their configurations and their defining equations. 
In this paper we extend some of the results obtained in \cite{BGV} in the case $n=2$ to arbitrary dimension. For the convenience of the reader, we summarize here our main contributions:
\begin{itemize}
\item In \Cref{thm: resolution of eigenscheme ideals} we explicitly compute the minimal graded free resolution of the defining ideal of a general eigenscheme. 
\item In \Cref{prop: equations partially symmetric} and \Cref{thm: equations symmetric} we characterize determinantal equations of eigenschemes of tensors and symmetric tensors, thereby giving a positive solution to \cite[Conjecture 3.13]{BGV}. These theorems allow us to design an algorithm which checks if a tuple of homogeneous polynomials is the tuple of determinantal equations of an eigenscheme. The second algorithm checks whether, given a set of points $Z\subseteq\p^n$, there exists a tensor $T$, possibly symmetric, whose eigenscheme contains $Z$. If the answer is positive, it explicitly reconstructs $T$. 
\item We prove that the variety parametrizing eigenschemes of symmetric tensors is rational and we give a birational parametrization. This is the content of \Cref{pro: Eig e' razionale}.
\item Finally, in \Cref{pro: no kd points on a curve} we give a geometric necessary condition for a $0$-dimensional scheme to be the eigenscheme of a tensor. This gives a partial answer to \cite[Question 16]{RS}.
\end{itemize}

\ackn 
Beorchia (ORCID 0000-0003-3681-9045) is a member of GNSAGA of INdAM and is supported by the fund Universit\`a degli Studi di Trieste - FRA 2022, by the MIUR Excellence Department Project awarded to the DMG of Trieste, 2018-2023, and by the Italian MIUR funds, PRIN project Moduli Theory and Birational Classification (2017), P.i. U. Bruzzo.

Galuppi (ORCID 0000-0001-5630-5389) is supported by the National Science Center, Poland, project ``Complex contact manifolds and geometry of secants'', 2017/26/E/ST1/00231.

Venturello (ORCID 0000-0002-6797-5270) is supported by the G\"{o}ran Gustafsson foundation.

\section{Definitions and first results}\label{section: syzygies}
In this section we define eigenschemes and discuss their algebraic invariants. 

\begin{definition}
\label{def:eigenscheme}
Let $T= (g_0, \ldots, g_n)$ be a partially symmetric tensor. The \emph{eigenscheme} of $T$ is the closed subscheme $E(T)\subseteq\p^n$ defined by the $2 \times 2$ minors of
	\[\left(
	\begin{matrix}
	x_0 & x_1 & \dots & x_n \\
	g_0 & g_1 & \dots & g_n
	\end{matrix}\right).
	\]
These minors are called \emph{determinantal generators} of the ideal $I_{E(T)}$. When $T$ is symmetric, that is $T$ is identified with a homogeneous polynomial $f$, we have $g_i=\de_if$ and we denote its eigenscheme by $E(f)$.
\end{definition}
\begin{rmk}\label{rmk: scelta di una base ortonormale}\label{rmk: dipende dalla scelta di una base}
The property of being an eigenvector 
is not invariant under linear transformations. For this reason, when we speak of \textquotedblleft the eigenpoints of a tensor $T$\textquotedblright\ we are abusing terminology. In order to avoid ambiguities, we choose once and for all a basis consisting of real vectors  
that are orthonormal with respect to the Euclidean scalar product in $\R^{n+1}$. With this choice we identify the space $\Sym^d\C^{n+1}$ of symmetric tensors with the space $\C[x_0,\dots,x_n]_d$ of degree $d$ homogeneous polynomials. In the same way we identify $\Sym^{d-1}\C^{n+1}\otimes\C^{n+1}$ with $(\Sym^{d-1}\C^{n+1})^{\oplus (n+1)}$. In order to remember this important choice, we always think of partially symmetric tensors as tuples of polynomials. Once the basis is chosen, we can safely compute eigenvectors as in \eqref{eq: definizione di autovettore}.
\end{rmk}

It is possible to define the eigenscheme of a tensor which is not partially symmetric, as in \cite[Definition 1.1]{CartSturm} or in \cite[Section 1]{ASS}. However, with the choice of a basis described in Remark \ref{rmk: dipende dalla scelta di una base}, it is apparent from the definition that for every tensor $T$ there exists a partially symmetric tensor $T^\prime$ such that $E(T^\prime)=E(T)$. For this reason, it is not restrictive to consider partially symmetric tensors.

\begin{definition}
\label{def: phi} We define the integer $w(n,d)$ as
\[w(n,d)=\frac{(d-1)^{n+1}-1}{d-2},
\]
for every $n\geq 1$ and $d\geq 3$. When $d=2$, we set $w(n,2)=n+1$.
\end{definition}
The number $w(n,d)$ is the number of eigenpoints of a general tensor in $(\Sym^{d-1}\C^{n+1})^{\oplus (n+1)}$, as proven in \cite[Theorem 1.2]{CartSturm}.
\begin{thm}
    Let $n\geq 1$, $d\geq 2$ and let $T\in (\Sym^{d-1}\C^{n+1})^{\oplus (n+1)}$ be a general tensor. Then $E(T)\subseteq \p^{n}$ is a $0$-dimensional scheme of degree $w(n,d)$.
\end{thm}
Since eigenschemes have an algebraic definition, it is natural to study algebraic properties of their ideals. For instance, we are interested in their minimal free resolution. In general this can be quite complicated, because eigenschemes can have positive dimension and they can even have components of different dimensions. However, the general tensor has a $0$-dimensional eigenscheme, and the same is true for the general symmetric tensor - see \cite[Section 5.2]{Abo}. For this reason we focus on the minimal graded free resolution of the ideal of a $0$-dimensional eigenscheme. We accomplish that in Theorem \ref{thm: resolution of eigenscheme ideals}.

Let $R = \mathbb{C}[x_0,\dots,x_n]$, let $g_0,\dots,g_n\in R_{d-1}$ and consider the matrix
\begin{equation}\label{main_matrix}
	M = \begin{pmatrix}
			x_0 & x_1 & \dots & x_n\\
			g_0 & g_1 & \dots & g_n
		\end{pmatrix}.
\end{equation}
We regard $M$ as a map $R^{n+1}\to R(-1)\oplus R(-d+1)$ of free graded $R$-modules. In order to study the ideal of minors of $M$, we recall the definition of the Eagon-Northcott complex.

\begin{defn}\label{def:EN} Let $\alpha:F\to G$ a map of finitely generated $R$-modules, with 
$$
f=\rank(F)\mbox{ and }g=\rank(G).
$$
The \emph{Eagon-Northcott complex} $EN(\alpha)_\bullet$ is the complex defined by
	\[\begin{cases}
EN(\alpha)_1=\bigwedge^{g}G & \\
EN(\alpha)_i= (\Sym_{i-1}(G))^*\otimes\bigwedge^{i+g-1}F & \mbox{ for every }2\leq i \leq f-g+1,
	\end{cases}
	\]
with differentials as given in \cite[Section A2H]{Eisenbud}.
\end{defn}
In our setting $F=R^{n+1}$, $G=R(1)\oplus R(d-1)$ and $\alpha$ is the map defined by left multiplication with $M$. The shifts are such that the map $\alpha:F\to G$ with $(x_0,\dots,x_n)\mapsto M(x_0,\dots,x_n)^T$ is homogeneous and they induce a shift in the modules $EN(\alpha)_i$.

\begin{thm}\label{thm: resolution of eigenscheme ideals}
Let $n$ and $d$ be positive integers such that $d\ge2$ and let $R = \mathbb{C}[x_0,\dots,x_n]$. Let $I$ be the ideal generated by the $2\times 2$ minors of a matrix $M$ as in \eqref{main_matrix} and let $V(I)$ be the zero locus of $I$.
If $\dim V(I)=0$, then $I$ is saturated and the minimal graded free resolution
$$F_{n}\to \ldots \to F_1\to I\to 0$$
of $I$ as an $R$-module satisfies
	\[
		F_{i}\cong \bigoplus_{j=1}^{i} R(-j(d-2)-i-1)^{\binom{n+1}{i+1}}.
	\]
\end{thm}		
\begin{proof}
By hypothesis we have
$\codim V(I) = n = n+1-2+1$, so $V(I)$ is a determinantal subscheme of the expected codimension given by the Porteous' Formula - see \cite[Corollary 11]{Kempf}. 
By the Hochster-Eagon Theorem - see \cite{HE} or \cite[Theorem 18.18]{eisenvudcommutative} - $R/I$ is a Cohen-Macaulay ring, so $I$ is saturated. By \cite[A2.10.C]{Eisenbud}, the Eagon-Nothcott complex associated with $\alpha$ is exact and it gives a minimal free resolution of $I$.

Specifically, in our setting we have
    \begin{align*}
	(\Sym_{i-1}(R&(1)\oplus R(d-1)))^*=\bigoplus_{j=0}^{i-1}(\Sym_{i-1-j}(R(1)))^*\otimes(\Sym_{j}(R(d-1)))^*\\
	&=\bigoplus_{j=0}^{i-1}(\Sym_{i-1-j}(R)(i-1-j))^*(-1)\otimes(\Sym_{j}(R)(j(d-1)))^*\\
	&=\bigoplus_{j=0}^{i-1}(\Sym_{i-1-j}(R))^*(-i+1+j)\otimes(\Sym_{j}(R))^*(-j(d-1))\\
	&\cong \bigoplus_{j=0}^{i-1} R(-j(d-2)-i+1).
\end{align*}
The image of the map
\[
\bigwedge^{2} F \cong R^{\binom{n+1}{2}}\to \bigwedge^2 G \cong R(d)
\]
is the ideal $I$ of $2\times 2$ minors of a matrix representing $\alpha$ as in \Cref{def:EN}. In our application, this is the defining ideal of the eigenscheme of a partially symmetric tensor in $\Sym^{d-1}\mathbb{C}^{n+1}\otimes \mathbb{C}^{n+1}$. As we consider this ideal in a standard graded polynomial ring $R$, we shift all the modules in $EN(\alpha)$ by a factor of $-d$, i.e., we consider the complex $EN(\alpha)_\bullet\otimes R(-d)$. Its $i$-th piece is equal to
\begin{align*}
	EN(\alpha)_i\otimes R(-d)&\cong \bigoplus_{j=0}^{i-1} R(-j(d-2)-i+1-d)^{\binom{n+1}{i+1}}\\
	&=\bigoplus_{j=0}^{i-1} R(-(j+1)(d-2)-i-1)^{\binom{n+1}{i+1}}.\qedhere
\end{align*}
\end{proof}

\begin{rmk}
Notice that the defining ideal of an eigenscheme of a tensor, even of a symmetric tensor, may fail to be saturated, as shown in \cite[Remark 3.2]{BGV}. In this case the eigenscheme has dimension at least 1.
\end{rmk}
 
 \begin{rmk}\label{rmk: degrees of a resolution}
 The statement of Theorem \ref{thm: resolution of eigenscheme ideals} holds for any $0$-dimensional determinantal scheme given by the $2 \times 2$ minors of a $2 \times (n+1)$ matrix whose first row consists of linear entries and the second row of degree $d-1$ homogeneous forms.
 \end{rmk}
{As a direct consequence we can see in the following Corollary that any matrix whose maximal minors define a $0$-dimensional eigenscheme involves $n+1$ linearly independent linear entries in the first row. Indeed,
by \cite[Lemma 2.9]{Kreuzer}, any such a matrix is of size $2 \times (n+1)$, with homogeneous rows of degrees $1$ and $d-1$. The next result is a generalization of the analogous result for $n=2$ given in \cite[Proposition 5.2]{ASS}. 
	
	\begin{cor}
Let $d\ge 2$. Let $l_0,\dots,l_n\in \mC[x_0,\dots,x_n]_1$ and $h_0,\dots,h_n\in \mC[x_0,\dots,x_n]_{d-1}$. Let 
	$$
	N =\begin{pmatrix}
			l_0 & l_1 & \dots & l_n\\
			h_0 & h_1 & \dots & h_n
		\end{pmatrix}
	$$
and let $W\subseteq\p^n$ be the scheme defined by the $2\times 2$ minors of $N$. If $\dim W=0$, then $W$ is the eigenscheme of a tensor $T\in(\C^{n+1})^{\otimes d}$ if and only if the linear forms $l_0, l_1,\dots,l_n$ are linearly independent.
	\end{cor}
		} 
\begin{proof}
If $l_0, l_1,\dots,l_n$ are linearly independent, then there exists $B \in \GL (n+1,\mC)$ such that 
	    $$
	    N'=N \cdot B=\begin{pmatrix}
			x_0 & x_1 & \dots & x_n\\
			b_0 & b_1 & \dots & b_n
		\end{pmatrix},
		$$
		where each $b_j \in \mC[x_0,\dots,x_n]_{d-1}$ is a linear combination of the polynomials $h_k$. The zero set of the $2\times 2$ minors of $N'$ is $W$, and by definition it is the eigenscheme of the partially symmetric tensor $(b_0,
		b_1,\dots, b_n)$.
Conversely, assume that there exists a partially symmetric tensor $T=(g_0,g_1,\dots,g_n)$ such that $E(T)=W$. As the $2\times 2$ minors of $N$ define a $0$-dimensional eigenscheme, the Eagon-Northcott
	complex associated with $N$ is exact and. By Remark \ref{rmk: degrees of a resolution}, its first terms are given by
	$$
	 R(-d-1)^{\binom{n+1}{3}} \oplus R(-2d+1)^{\binom{n+1}{3}} 
	 \xrightarrow{d_2^N} R(-d)^{\binom{n+1}{2}} \to R 
	$$
	
By the definition of the first syzygy map $d_2^N: F_2^N \to F_1^N$,
	it is possible to see that $d_2^N$ is defined by a matrix which is linear in the entries of $N$. More precisely, by setting 
	$$
	m_{ij}=l_i h_j-l_jh_i,
	$$
	the generating syzygies are given by
	$$
	l_i m_{jk}-l_j m_{ik}+l_k m_{ij}=0\mbox{ and } h_i m_{jk}-h_j m_{ik}+h_k m_{ij}=0
	$$
for every $0\leq i<j<k\leq n$. In particular, the first $\binom{n+1}{3}$ columns of $d_2^N$ contain only the forms $l_i$ and $0$. Since $W$ is an eigenscheme by assumption, its ideal is generated also by the $2\times 2$ minors of a matrix $M$ as in \eqref{main_matrix}, so by Theorem \ref{thm: resolution of eigenscheme ideals} the Eagon-Northcott complex associated with $M$ gives also a minimal free resolution of $I_W$.
By \cite[Theorem 20.2]{eisenvudcommutative}, the minimal free resolution of a projective module is unique up to isomorphism of complexes. This means that there is a matrix $ H \in M(\binom{n+1}{3},R_{d-2})$ with degree $d-2$ forms as entries, and there are three invertible matrices $A\in \GL(\binom{n+1}{2},\mC)$ and $B,C\in \GL(\binom{n+1}{3},\mC)$ such that if we call
	$$
	G=\begin{pmatrix}
	B
  & \rvline & H \\
\hline
  0 & \rvline &
  C
	\end{pmatrix},
	$$
then we have $d_2^M = A \cdot d_2^N \cdot G$. In particular, the linear entries $x_0, \dots, x_n$ of $d_2^M$	are linear combinations of 
$l_0, \dots, l_n$. So these forms span the whole of $\mC[x_0, \dots, x_n]_1$, thus they are linearly independent.
		\end{proof}

\section{Characterization of determinantal equations}\label{sec: algorithms}
In this section we give a characterization of tuples of homogeneous polynomials which are determinantal generators of the homogeneous ideal of an eigenscheme, as introduced in Definition \ref{def:eigenscheme}. 

When $n=2$, such a characterization has been found in \cite[Section 3]{BGV} for both general and symmetric tensors. The main ingredients in our proof are two prominent chain complexes in commutative algebra, namely the Koszul complex and the de Rham complex. We begin with the general case.

\begin{thm}\label{prop: equations partially symmetric}
Let $R=\C[x_0,\dots,x_n]$. A tuple $(f_{ij}~:~ 0\leq i<j\leq n)\in R_d^{\binom{n+1}{2}}$ of homogeneous polynomials is the tuple of determinantal equations of the eigenscheme of a tensor if and only if
\begin{equation}\label{eq: equation partial}
    x_if_{jk} - x_j f_{ik} + x_k f_{ij} = 0	 
\end{equation}
for every $0\leq i<j<k\leq n$.
\end{thm}
\begin{proof} By \cite[Section 17.2]{eisenvudcommutative}, the Koszul complex of $x_0,\dots,x_n$ starts with
\[
0\to R\xrightarrow{\alpha} 
 R^{\oplus n+1}\xrightarrow{\beta} \bigwedge^2 R^{\oplus n+1}\xrightarrow{\gamma} \bigwedge^3 R^{\oplus n+1} ,
\]
where
\begin{align*}
\alpha&(h)=x_0h+\ldots+x_nh,\\
\beta&(g_0,\ldots,g_n)=(x_ig_j-x_jg_i~:~ 0\leq i<j\leq n)\\
\gamma&((f_{ij})_{i<j})=(x_if_{jk} - x_j f_{ik} + x_k f_{ij}~:~ 0\leq i<j<k\leq n).
\end{align*}
The tuple $(f_{ij}~:~ 0\leq i<j\leq n)$ is the tuple of determinantal equations of the eigenscheme of a tensor if and only if $(f_{ij})_{i<j}\in\Ima(\beta)$. Since $x_0,\dots,x_n$ is regular sequence, the sequence is exact. In particular, $\Ima(\beta)=\ker(\gamma)$.
\end{proof}

\begin{rmk}\label{rmk: when two eigenschemes have same determinantal equations}
    We observe that exactness of the Koszul complex in the second term implies that
    two partially symmetric tensors $T=(g_0, \dots, g_n)$ and $T'=(g_0', \dots, g_n')$  
    determine the same determinantal equations for their eigenschemes if and only if there exist $c \in \mC^*$ and $h \in R_{d-2}$ such that 
$$
g'_k = c g_k +x_k h
$$
for every $k\in\{0, \dots, n\}$. Indeed, 
since $\Ima(\alpha) =\ker(\beta)$,
we have $T'-T=(x_0h,\dots,x_nh)$, for some $h\in R_{d-2}$.
    
\end{rmk}

As a consequence of the previous remark, we observe that the eigenschemes are an example of general determinantal schemes, which admit a positive dimensional family of defining matrices and they are not obtained by left and right multiplication by a scalar invertible matrix. This is not the case for general determinantal schemes of dimension $\ge 1$, which are defined by the maximal minors of a matrix, as shown in \cite[Main Theorem]{FF}.

Now we are interested in a more refined version of Proposition \ref{prop: equations partially symmetric}. Namely, we would like a similar characterization for determinantal equations of symmetric tensors. In particular, we will prove \cite[Conjecture 3.13]{BGV}. Our argument relies on the exactness of the affine de Rham complex, which we briefly recall.
\begin{defn}
\label{def: de Rham complex}
Let $R = \mathbb{C}[x_0,\dots,x_n]$. The \emph{de Rham complex} of $\mathbb{A}^{n+1}$ is the complex of $R$-modules defined as
	\[
	0\to \Omega^0_{\mathbb{A}^{n+1}}\to\Omega^1_{\mathbb{A}^{n+1}}\to\cdots\to\Omega^{n+1}_{\mathbb{A}^{n+1}}\to 0,
	\]
	where $\Omega^k_{\mathbb{A}^{n+1}}$ is the module of algebraic differential $k$-forms over $\mathbb{A}^{n+1}$. By letting $dx_0,\dots,dx_n$ be an $R$-basis for $\Omega^1_{\mathbb{A}^{n+1}}$ we can write a $k$-form $\omega = \sum_{I\in\binom{[n+1]}{k}}f_I dx_I\in \Omega^k_{\mathbb{A}^{n+1}}$, with $f_I\in R$ and with $dx_I=dx_{i_1}\wedge\cdots\wedge dx_{i_k}$ for $I = \{i_1,\dots,i_k\}$. The differentials of the de Rham complex are given by exterior derivatives, i.e.
	\[
		d\omega = \sum_{I\in\binom{[n+1]}{k}}\sum_{j=0}^{n}\de_j f_I dx_j\wedge dx_I.
	\]\end{defn}
We are interested in the differential 
	\begin{align*}
d:\Omega^1_{\mathbb{A}^{n+1}}&\to\Omega^2_{\mathbb{A}^{n+1}}\\
		\sum_{i=0}^n f_i dx_i&\mapsto \sum_{i=0}^n\sum_{j=0}^n \de_j f_idx_j\wedge dx_i = \sum_{0\leq i<j\leq n} (\de_if_j-\de_jf_i)dx_i\wedge dx_j.
	\end{align*}
	As the de Rham complex of $\mathbb{A}^{n+1}$ is exact everywhere but in degree 0 (this follows, from instance, from de Rham's theorem relating the cohomology of de Rham complex with the singular cohomology of $\mathbb{A}^{n+1}$), we conclude that whenever a tuple of $n+1$ polynomials $(f_0,\dots,f_n)$ satisfy the relations $\de_if_j-\de_jf_i$ for every $0\leq i<j\leq n$, then there exists an element of $\Omega^0_{\mathbb{A}^{n+1}}\cong R$ such that $\nabla g = (f_0,\dots,f_n)$. If moreover the polynomials $f_0,\dots,f_n$ are homogeneous of degree $d-1$, then by Euler formula $g$ is homogeneous of degree $d$.

We are now in position to characterize determinantal generators of symmetric tensors. The next result proves \cite[Conjecture 3.13]{BGV}.
	\begin{thm}\label{thm: equations symmetric} Let $d\ge 2$,  $n\ge 2$ and let $R=\C[x_0,\dots,x_n]$. A tuple $(f_{ij}~:~ 0\leq i<j\leq n)\in R_d^{\binom{n+1}{2}}$ is the tuple of determinantal equations of the eigenscheme of a polynomial in  $R_d$ if and only if 
		\begin{equation}\label{eq: x_if_jk-...}
			x_if_{jk}-x_jf_{ik}+x_kf_{ij} = 0
		\end{equation}
		and 
		\begin{equation}\label{eq: de_if_jk-...}
		\de_if_{jk}-\de_jf_{ik}+\de_kf_{ij} = 0
		\end{equation}
		for every $0\leq i<j<k\leq n$.
	\end{thm}
Note that the equations labelled with \eqref{eq: x_if_jk-...} are precisely the equations in \eqref{eq: equation partial}.	
\begin{proof}
It is immediate to check that the determinantal equations of the autoscheme of a polynomial satisfy \eqref{eq: x_if_jk-...} and \eqref{eq: de_if_jk-...}. Now we assume that the two equations hold and we have to prove the existence of a polynomial $f\in R_d$ such that $x_i\de_j f-x_j\de_i f=f_{ij}$. Assume that \eqref{eq: x_if_jk-...} holds for every $(i,j,k)$. By \Cref{prop: equations partially symmetric} there exist polynomials $g_0,\dots,g_n\in R_{d-1}$ such that $f_{ij}=x_ig_j-x_jg_i$ for every $0\leq i<j\leq n$.	If \eqref{eq: de_if_jk-...} holds for a fixed triple $(i,j,k)$, then
		\begin{align}
0&=\de_if_{jk}-\de_jf_{ik}+\de_kf_{ij}\nonumber\\
&=\de_i(x_jg_k-x_kg_j)-\de_j(x_ig_k-x_kg_i)+\de_k(x_ig_j-x_jg_i)\nonumber\\
&=x_i(-\de_jg_k+\de_kg_j)-x_j(-\de_ig_k+\de_kg_i)+x_k(-\de_ig_j+\de_jg_i).\label{eq: deigj-dejgi equation}
		\end{align}
		If $d=2$, then $\deg(-\de_jg_k+\de_kg_j)=\deg(-\de_ig_k+\de_kg_i)=\deg(-\de_ig_j+\de_jg_i)=0$. The last equation above implies then that $-\de_jg_k+\de_kg_j=-\de_ig_k+\de_kg_i=-\de_ig_j+\de_jg_i=0$. By repeating the same step with other triples $(i,j,k)$ we obtain that $\de_jg_i-\de_ig_j=0$ for every $0\leq i<j\leq n$. By what we discussed above, this implies that there exists $f\in R_d$, such that $\nabla f = (g_0,\dots,g_n)$. Substituting the identities $\de_i f = g_i $ in $f_{ij} = x_ig_j-x_jg_i$ we obtain that $f_{ij} = x_i\de_jf-x_j\de_if$ which proves the claim.\\
		Assume the claim holds for every degree smaller than $d$. Then \eqref{eq: deigj-dejgi equation} shows that the tuple $(\de_ig_j-\de_jg_i: 0\leq i<j\leq n)\in R_{d-2}$ satisfy \eqref{eq: x_if_jk-...}. By Schwartz theorem, it also satisfies \eqref{eq: de_if_jk-...}. By induction, there exists a polynomial $h\in R_{d-2}$ such that 
		\begin{equation*}
			(\de_ig_j-\de_jg_i)-(\de_i x_jh-\de_jx_ih)=\de_i(g_j-x_jh)-\de_j(g_i-x_ih)=0\\
		\end{equation*}
		for every $0\leq i<j\leq n$. This implies the existence on a polynomial $f\in R_{d}$ such that $g_i-x_ih=\de_if$. Since $\de_ix_jf-\de_jx_if=x_j(g_i-x_ih)-x_i(g_j-x_jh)=x_jg_i-x_ig_j=f_{ij}$ we conclude.
\end{proof}

\Cref{prop: equations partially symmetric} and \Cref{thm: equations symmetric} give an effective algorithm to test whether $\binom{n+1}{2}$ homogeneous forms $h_1,\dots,h_{\binom{n+1}{2}}$ of degree $d$ are the defining equations of an eigenscheme (see \Cref{alg: algorithm1}). Note that if conditions \eqref{eq: x_if_jk-...} or \eqref{eq: de_if_jk-...} are not satisfied by the triples of polynomials $h_i$ in any order, we cannot conclude that these polynomials do not define the eigenscheme of a tensor. Indeed, there might be a change of basis in the linear system of the polynomials $h_i$ for which the new basis satisfies conditions \eqref{eq: x_if_jk-...} or \eqref{eq: de_if_jk-...}. 

\begin{algorithm}[H]
\begin{flushleft}
    \vspace{2pt}
	\textbf{Input:} $h_1,\dots,h_{\binom{n+1}{2}}\in\C[x_0,\dots,x_n]_d$. \\
	\textbf{Step 1:} Set $M=(m_{ij})$ a $\binom{n+1}{2}\times \binom{n+1}{2}$ matrix of indeterminates.\\
	\textbf{Step 2:} Let $(f_{ij} ~: ~ 1\leq i<j\leq n+1)=M\cdot (h_1\c
	dots h_{\binom{n+1}{2}})^\top$. The coefficients of the polynomials $f_{ij}$ are linear combinations of the $m_{ij}$.\\
	\textbf{Step 3:} Consider the system of $\binom{n+1}{3}$ linear equations in the $m_{ij}$ given by \eqref{eq: equation partial}.\\
	\hspace{20pt} \textbf{If} the system has a nontrivial solutions\\
	\hspace{40pt}\textbf{return} False: $h_1,\dots,h_{\binom{n+1}{2}}$ do not generate the eigenscheme of a partially symmetric tensor.\\
	\hspace{20pt}\textbf{Else}\\ \hspace{40pt}\textbf{return} True:  $h_1,\dots,h_{\binom{n+1}{2}}$ generate the eigenscheme of a partially symmetric tensor.\\
	\textbf{Step 4:} If  Step 3 returned True, consider the system of $\binom{n+1}{3}$ linear equations in the $m_{ij}$ given by \eqref{eq: de_if_jk-...}.\\
	\hspace{20pt}\textbf{If} the system has no nontrivial solution\\
	\hspace{40pt}\textbf{return} False: $h_1,\dots,h_{\binom{n+1}{2}}$ do not generate the eigenscheme of a symmetric tensor.\\ \hspace{20pt} \textbf{Else}\\
	\hspace{40pt} \textbf{return} True: $h_1,\dots,h_{\binom{n+1}{2}}$ generate the eigenscheme of a symmetric tensor.
	\caption{Test if homogeneous forms are defining equations of an eigenscheme}
	\label{alg: algorithm1}
\end{flushleft}	
\end{algorithm}
Another computational problem one is naturally concerned with is the following: given $m$ distinct points in $\p^n$, is there an eigenscheme containing them? In \Cref{alg: algorithm2} we propose a straightforward procedure to test this, by simply considering a generic partially symmetric or symmetric tensor and prescribing the vanishing of the minors of the associated matrix \eqref{main_matrix} on the points. This yields a system of $m\binom{n+1}{2}\binom{n+d}{d}$ linear equations in $(n+1)\binom{n+d-1}{d-1}$ (in the partially symmetric case) or $\binom{n+d}{d}$ (in the symmetric case) variables. From a solution of this system, a witness tensor can be easily reconstructed.
\begin{algorithm}[H]
\begin{flushleft}
    \vspace{2pt}
	\textbf{Input:} $p_1,\dots,p_m\in\p^n$, $d\geq 2$.\\
	\textbf{Step 1a:} For $i\in\{0,\dots, n\}$, let $$g_i=\sum_{\alpha}\lambda_{i,\alpha}x^\alpha,$$ where $x^\alpha$ ranges among all degree $d-1$ monomials in the variables $x_0,\ldots,x_n$. 	In words, each $g_i$ is a polynomial of degree $d-1$ with distinct unknowns $\lambda_{i,\alpha}$ as coefficients.\\
	\textbf{Step 2a:} Consider the system of $m\binom{n+1}{2}\binom{n+d}{d}$ linear equations in the variables $\lambda_{i,\alpha}$ given by $x_ig_j(p_k)-x_jg_i(p_k)=0$ for every $0\leq i<j\leq n$ and for every $k\in\{1,\dots,m\}$.\\
	\hspace{20pt} \textbf{If} the system has a non trivial solution\\
	\hspace{40pt} \textbf{return } True: $p_1,\dots,p_m$ are contained in the eigenscheme of a partially symmetric tensor in $\text{Sym}^{d-1}(\C^{n+1})\otimes\C^{n+1}$. A witness tensor is given by $(g_0,\dots,g_n)$, after evaluating the coefficients of each $g_i$ in a solution of the system.\\
	\textbf{Step 1b:} Let $$f=\sum_{\alpha}\lambda_{i,\alpha}x^\alpha,$$ where $x^\alpha$ ranges among all degree $d$ monomials in the variables $x_0,\ldots,x_n$.\\
	\textbf{Step 2b:} Consider the system of $m\binom{n+1}{2}\binom{n+d}{d}$ linear equations in the variables $\lambda_{i,\alpha}$ given by $x_i\de_j f(p_k)-x_j\de_i f(p_k)=0$ for every $0\leq i<j\leq n$ and for every $k\in\{1,\dots,m\}$.\\
	\hspace{20pt} \textbf{If} the system has a non trivial solution\\
	\hspace{40pt} \textbf{return } True: $p_1,\dots,p_m$ are contained in the eigenscheme of a symmetric tensor in $\text{Sym}^{d}(\C^{n+1})$. A witness tensor is given by $f$, after evaluating the coefficients in a solution of the system. 
	
	\caption{Test if points are contained in an eigenscheme of a tensor of fixed degree}
	\label{alg: algorithm2}
\end{flushleft}	
\end{algorithm}

\section{The variety of eigenpoints}\label{section: eigenvariety}
Now we take a more geometric approach and we turn our attention to the variety parametrizing configurations of eigenpoints of tensors. Let $(\p^n)^{w(n,d)}$ be the symmetric power of $\p^n$ with itself $w(n,d)$ times, where $w(n,d)$ is the integer in \Cref{def: phi}. In other words, $(\p^n)^{w(n,d)}$ parametrizes sets of $w(n,d)$ points in $\p^n$. Since $w(n,d)$ is precisely the degree of the eigenscheme of a generic tensor, following \cite[Section 5]{ASS} we define the {\it variety of eigenconfigurations} $
\Eig_{n,d} \subseteq
(\p^n)^{w(n,d)}
$ to be the Zariski closure of the set of all eigenconfigurations of  
tensors in 
$(\C^{n+1})^{\otimes d}$.

The dimension of $\Eig_{2,d}$ has been determined in \cite[Theorem 5.5]{ASS}. Thanks to Remark \ref{rmk: when two eigenschemes have same determinantal equations}, we can give an upper bound for such a dimension in general.

\begin{prop}
\label{prop: dimension of the eigenvariety}
For any $n \ge 2$ and $d \ge 3$ we have
$$
\dim \Eig_{n,d} \le \dfrac{n(n+d)}{d-1}\binom{n+d-2}{n} -1 .
$$
\end{prop}
\begin{proof}
Consider the dominant rational map
$$
\eta _{n,d} : \p(\mC[x_0, \dots, x_n]_{d-1})
\times \dots \times \p(\mC[x_0, \dots, x_n]_{d-1})
\dasharrow \Eig_{n,d},
$$
with $\eta_{n,d} ([g_0], \dots, [g_n]) = E(T)$ and $T=(g_0, \dots, g_n)$. By Remark \ref{rmk: when two eigenschemes have same determinantal equations}, the general fiber of
$\eta_{n,d}$ contains a copy of $\p(\mC[x_0, \dots, x_n]_{d-2})$. Since the (affine) dimension of $\Eig_{n,d}$ is equal to the dimension of the source of the map $\eta_{n,d}$ minus that of the generic fiber, we have projectively that
\[
    \dim \Eig_{n,d} \le (n+1)\binom{n+d-1}{n}-\binom{n+d-2}{n}-1,
\]
and the claim follows.
\end{proof}

By \cite[Theorem 5.5]{ASS} we have $\dim \Eig_{2,d} = d^2+2d-1$, so the bound given by \Cref{prop: dimension of the eigenvariety} is sharp for $n=2$. We expect equality to hold for every $n\ge 2$, but at the moment we lack a characterization of all possible defining matrices for a given $0$-dimensional eigenscheme.

Now we turn to the symmetric case. The dimension of the corresponding variety has already been computed in
\cite[Theorem 1.1]{Turatti}. Here we prove that $\Eig_{n,d,\Sym}$ is rational and give a birational parametrization. Our argument is similar to the one presented in \cite[Remark 3.7]{BGV}.
\begin{definition} 
We set
\[
\begin{matrix}
\phi_{n,d}:&\p(\C[x_0,\ldots,x_n]_d)&\dashrightarrow & (\p^n)^{w(n,d)}\\
&f&\mapsto &E(f),
\end{matrix}
\]
and we define $\Eig_{n,d,\Sym}$ to be the closure of the image of $\phi_{n,d}$.

\end{definition}
Using \cite[Theorem 1.2]{CartSturm}, we obtain that $\phi_{n,d}$ is defined on all forms $f$ such that $E(f)$ is reduced and $0$-dimensional.

Let $\alpha$ be the projectivization of the linear map
\[\begin{matrix}
\C[x_0,\ldots,x_n]_d &\to &(\C[x_0,\ldots,x_n]_d)^{\oplus\binom{n+1}{2}}\\
f&\mapsto &(x_i\de_j f-x_j\de_i f)_{i<j},
\end{matrix}\]
and define
\begin{equation*}\label{Td per esteso}
T_{n,d}= \overline {\alpha (\p(\C[x_0,\ldots,x_n]_d))} \subseteq\p((\C[x_0,\ldots,x_n]_d)^{\oplus \binom{n+1}{2}})
\end{equation*}
the closure of the image of $\alpha$. Then $T_{n,d}$ is a linear subspace of $\p((\C[x_0,\ldots,x_n]_d)^{\oplus \binom{n+1}{2}})$ and there is a commutative diagram
\begin{displaymath}
  \xymatrix{
 \p(\C[x_0,\ldots,x_n]_d)\ar@{-->}[rr]^{\phi_{n,d}}\ar[rd]_{\alpha} & & \Eig_{n,d,\Sym}\\
       & T_{n,d}\ar@{-->}[ru]_{\psi_{n,d}} &
}
\end{displaymath}
where $\psi_{n,d}(f_{ij}~:~ 0\leq i<j\leq n)$ is the set of points defined by the ideal $(f_{ij}~:~ i<j)$ in $\p^n$.
\begin{prop}
\label{pro: Eig e' razionale}
Let $n$ and $d$ be positive integers such that $d\ge 3$. The map $\psi_{n,d}:T_{n,d}\dashrightarrow\Eig_{n,d,\Sym}$ is birational. In particular, $\Eig_{n,d,\Sym}$ is an irreducible, rational variety.
\end{prop}
\begin{proof}
We distinguish two cases. If $d$ is odd, then $\alpha$ is an isomorphism by \cite[Lemma 2.8]{BGV}. Since $\phi_{n,d}$ is birational by \cite[Theorem 1.1]{Turatti}, $\psi_{n,d}$ is birational as well. Assume now that $d$ is even. The fibers of $\alpha$ are lines by \cite[Lemma 2.8]{BGV} and the general fiber of $\phi_{n,d}$ is a line by \cite[Theorem 1.1]{Turatti}. This implies that $\psi_{n,d}$ is generically injective.
\end{proof}

\begin{rmk}
Another description of the eigenvariety was found in \cite[Section 5]{CGKS}. The authors show that $\Eig_{3,3,\Sym}$ is birational to a linear section of the Grassmanian $\G(3,14)$, but their argument can be easily adapted to different values of $n$ and $d$.
\end{rmk}

\section{Generalized Laguerre maps}\label{section: laguerre}
The goal of this section is to give a necessary condition for a $0$-dimensional scheme to be the eigenscheme of a tensor. We accomplish this task in Proposition \ref{pro: no kd points on a curve}. Our argument relies on \cite[Section 5]{BGV} and consists in studying the fibers of the Laguerre map associated with a partially symmetric tensor.

\begin{defn}\label{definition: laguerre map}Let $n\ge 2$ and $d\ge 3$, and
let $T=(g_0,\dots,g_n)\in\Sym^{d-1}\C^{n+1}\otimes\C^{n+1}$ be a partially symmetric tensor. For every $0\le i<j\le n$, let $f_{ij}=x_i g_j -x_j g_i$ be the $2\times 2$ minors of the matrix
\eqref{main_matrix}.
The {\it Laguerre map} associated with $T$ is the map 
$$
\lambda_T: \mP^n \dasharrow \mP(\Lambda ^2 \C^{n+1})
$$
given by $\lambda_T (P) = (f_{ij}(P):0\le i<j\le n)$. It is defined on the open subset $\p^n\setminus E(T)$.
\end{defn}

\begin{rmk}\label{rmk: contracted}
Let us take a closer look at the image of $\lambda_T$. It is easy to see that $\lambda_T (\mP^n)\subsetneq\mP(\Lambda ^2 \C^{n+1})$. Indeed, the coordinates of $\lambda_T(P)$ are the Pl\"ucker coordinates of the line in $\p^n$ joining $P$ and $(g_0(P):\ldots:g_n(P))$. Therefore $\lambda_T (\mP^n) \subseteq \mG(1,n)$.

By definition, the ideal of $E(T)$ is generated by the polynomials $f_{ij}$, hence the closure of the graph of $\lambda_T$ is a model of the blow-up ${\Bl}_{E(T)} \mP^n \subseteq\mP^n \times \mG(1,n)$. To determine its equations, observe that a point of $\mG(1,n)$ with Pl\"ucker coordinates corresponding to the $ 2 \times 2$ minors of a rank 2 matrix
\begin{equation*}\label{eq: ab}
\left(
	\begin{matrix}
	a_0 & a_1 & \dots & a_n \\
	b_0 & b_1 & \dots  & b_n\\
	\end{matrix}\right)
	\end{equation*}
lies in the image of $\lambda_T$ if and only if there exists a point $P=(P_0:P_1: \ldots :P_n) \in \mP^n$ such that 
$$
\rank\left(
	\begin{matrix}
	a_0 & a_1 & \dots & a_n \\
	b_0 & b_1 & \dots & b_n\\
	P_0 & P_1 & \dots & P_n \\
g_0(P) & g_1 (P)& \dots & g_n(P)
	\end{matrix}\right)=2.
	$$
Therefore the equations of ${\Bl}_{E(T)} \mP^n$ in $\mP^n \times \mG(1,n)$ are given by the 
$3 \times 3$ minors. Since the first two rows are linearly independent, by Kronecker's theorem on the rank of a matrix it is enough to consider the $3 \times 3$ minors obtained by deleting the third row 
and the  
$3 \times 3$ minors obtained deleting the fourth row. In particular, there are $\binom{n+1}{3}$ equations of bidegree $(1,1)$ and $\binom{n+1}{3}$ equations of bidegree $(d-1,1)$. 
\end{rmk}

We now want to study the fibers of $\lambda_T$. Let 
$$
p_1 : {\Bl}_{E(T)} \mP^n \to \p^n~\mbox{ and }~p_2 : {\Bl}_{E(T)} \mP^n \to \mG(1,n)
$$
be the projections. We observe that $\lambda_T^{-1}(\omega)=p_1(p_2^{-1}(\omega))$ for every $\omega\in\mG(1,n)$. Hence it suffices to look at fibers of $p_2$.
\begin{prop}\label{pro: fibre} Let $T=(g_0,\dots,g_n)\in\Sym^{d-1}\C^{n+1}\otimes \C^{n+1}$ be a partially symmetric tensor. Then every fiber of $\lambda_T: \mP^n \dasharrow \mG(1,n)$ is contained in a line in $\mP^n$.
\end{prop}
\begin{proof}
Let $\omega\in\Ima(\lambda_T) \subseteq \mG(1,n)$. Then there exist $Q_0,\ldots,Q_n\in\C$ such that the homogeneous coordinates of $\omega$ are the $2\times 2$ minors of the matrix \eqref{main_matrix} evaluated at $Q=(Q_0:\ldots:Q_n)$. In other words, the Pl\"ucker coordinates of $\omega$ are $(\bar p_{ij}~:~ 0\le i<j\le n)$, where $\bar p_{ij}=Q_ig_j(Q)-Q_jg_i(Q)\in\C$. The fiber $p_2 ^{-1} (\omega)$ is contained in the subvariety of $\p^n\times\mG(1,n)$ whose equations in $P_0,\dots,P_n$ are the $3\times 3$ minors of the matrix
$$
\left(
	\begin{matrix}
	Q_0 & Q_1 & \dots & Q_n \\
	g_0(Q) & g_1(Q) & \dots & g_n(Q)\\
	P_0 & P_1 & \dots & P_n \\
	\end{matrix}\right).
	$$
For all $0\le i<j<k\le n$, we get
\begin{equation}\label{eq:lineari dopo la proiezione di p1}
\bar p_{ij} P_k - \bar p_{ik} P_j + \bar p_{jk} P_i=0.\end{equation}
Once we project on the first factor, we see that the fiber $\lambda_T^{-1}(\omega)=p_1(p_2^{-1}(\omega))\subseteq\p^n$ is contained in the linear subspace defined by the $\binom{n+1}{3}$ linear equations \eqref{eq:lineari dopo la proiezione di p1} in the variables $P_0,\dots,P_n$.  It is not difficult to see that the coefficient matrix $A_\omega$ of this linear system gives the linear map
$\mC^{n+1} \to \bigwedge^3 \mC^{n+1}$ defined by$$A_\omega (v)=\omega \wedge v.
	$$
For instance, for $n=3$ we get
$$
A_{\omega}=\left(
	\begin{matrix}
	\bar p_{12} &-\bar p_{02} &\bar p_{01}&0\\
	\bar p_{13} & -\bar p_{03} & 0 & \bar p_{01}\\
	\bar p_{23} &0  &-\bar p_{03}  & \bar p_{02}\\
	0 &\bar p_{23}  &-\bar p_{13}  &\bar p_{12} \\
	\end{matrix}\right).
$$	
By recalling that a 2-vector $\omega$ is decomposable if and only if $\dim \ker A_{\omega}=2$, we have that $\rank A_{\omega}=n-1$. It follows that $\lambda_T^{-1}(\omega)$ is contained in an intersection of $n-1$ hyperplanes in $\mP^n$, which is a line.
	\end{proof}

Equations \eqref{eq:lineari dopo la proiezione di p1} determine the projectivization of the tautological bundle of $\mG(1,n)$ inside $\p^n\times\mG(1,n)$.

Proposition \ref{pro: fibre} allows us to prove the main result of this section, thereby generalizing \cite[Theorem 5.5]{BGV}.
\begin{prop}\label{pro: no kd points on a curve}
Let $n\ge 2$, $d\ge 2$ and let $T\in(\C^{n+1})^{\otimes d}$ be a tensor such that $\dim E(T)=0$. Then for every $k\in\{2,\dots,d-1\}$, no degree $kd$ subscheme of $E(T)$ is contained in an irreducible curve of degree $k$. Moreover, no degree $d+1$ subscheme of $E(T)$ is contained in a line.
\end{prop}
\begin{proof}
As usual, we can assume that $T\in\Sym^{d-1}\C^{n+1}\otimes \C^{n+1}$ is partially symmetric. In \cite[Proposition 2.6(1)]{BGV} we can see how $E(T)$ cannot contain $d+1$ collinear points, and the argument holds verbatim for a $0$-dimensional scheme of degree $d+1$. Assume that there exists a degree $kd$ subscheme of $E(T)$ contained in an irreducible curve $C \subseteq\mP^n$ of degree $k$. If we call $\LL=\langle f_{ij}~:~ 0\le i<j\le n\rangle$ the linear system on $\p^n$ associated with the Laguerre map $\lambda_T$, then we can interpret $\lambda_T:\p^n\dashrightarrow\p(\LL^*)$ as the map assigning to a point $P\in \mP^n \setminus E(T)$ the hyperplane
$$\LL_P=\{D\in\LL~:~ P\in D\}\subseteq\LL.$$
If $P\in C$, then every $D\in\LL_P$ contains $E(T)\cup\{P\}$, so it intersects $C$ with multiplicity at least $kd+1>\deg(C)\cdot\deg(D)$. Since $C$ is irreducible, we deduce that $C\subseteq D$ by  B\'ezout's Theorem. Hence
\[\lambda_T(P)=\LL_P=\{D\in\LL~:~ C\subseteq D\}
\]
is constant for every $P\in C\setminus E(T)$. It follows that $C$ is contracted by $\lambda_T$, and this implies that $k=1$ by Proposition \ref{pro: fibre}.
\end{proof}

The bound on collinear points given by Proposition \ref{pro: no kd points on a curve} is sharp. Indeed, it is easy to find order $d$ tensors whose eigenscheme contains $d$ points on a line, see for instance \cite[Example 4.2]{CGKS}.

\bibliographystyle{plain}

\end{document}